\documentclass{article}%
\usepackage{graphicx}
\usepackage{amsmath}
\usepackage{amsfonts}
\usepackage{amssymb}
\usepackage{setspace}
\usepackage[hmargin=3.5cm,vmargin=3.5cm]{geometry}%
\providecommand{\U}[1]{\protect\rule{.1in}{.1in}}
\providecommand{\U}[1]{\protect\rule{.1in}{.1in}}
\providecommand{\U}[1]{\protect\rule{.1in}{.1in}}
\providecommand{\U}[1]{\protect\rule{.1in}{.1in}}
\providecommand{\U}[1]{\protect\rule{.1in}{.1in}}
\providecommand{\U}[1]{\protect\rule{.1in}{.1in}}
\providecommand{\U}[1]{\protect\rule{.1in}{.1in}}
\providecommand{\U}[1]{\protect\rule{.1in}{.1in}}
\providecommand{\U}[1]{\protect\rule{.1in}{.1in}}
\providecommand{\U}[1]{\protect\rule{.1in}{.1in}}
\providecommand{\U}[1]{\protect\rule{.1in}{.1in}}
\providecommand{\U}[1]{\protect\rule{.1in}{.1in}}
\providecommand{\U}[1]{\protect\rule{.1in}{.1in}}
\providecommand{\U}[1]{\protect\rule{.1in}{.1in}}
\providecommand{\U}[1]{\protect\rule{.1in}{.1in}}
\providecommand{\U}[1]{\protect\rule{.1in}{.1in}}
\providecommand{\U}[1]{\protect\rule{.1in}{.1in}}
\providecommand{\U}[1]{\protect\rule{.1in}{.1in}}
\providecommand{\U}[1]{\protect\rule{.1in}{.1in}}
\providecommand{\U}[1]{\protect\rule{.1in}{.1in}}
\providecommand{\U}[1]{\protect\rule{.1in}{.1in}}
\providecommand{\U}[1]{\protect\rule{.1in}{.1in}}
\providecommand{\U}[1]{\protect\rule{.1in}{.1in}}
\providecommand{\U}[1]{\protect\rule{.1in}{.1in}}
\providecommand{\U}[1]{\protect\rule{.1in}{.1in}}
\providecommand{\U}[1]{\protect\rule{.1in}{.1in}}
\providecommand{\U}[1]{\protect\rule{.1in}{.1in}}
\providecommand{\U}[1]{\protect\rule{.1in}{.1in}}
\providecommand{\U}[1]{\protect\rule{.1in}{.1in}}
\providecommand{\U}[1]{\protect\rule{.1in}{.1in}}
\providecommand{\U}[1]{\protect\rule{.1in}{.1in}}
\providecommand{\U}[1]{\protect\rule{.1in}{.1in}}
\providecommand{\U}[1]{\protect\rule{.1in}{.1in}}
\providecommand{\U}[1]{\protect\rule{.1in}{.1in}}
\providecommand{\U}[1]{\protect\rule{.1in}{.1in}}
\providecommand{\U}[1]{\protect\rule{.1in}{.1in}}
\providecommand{\U}[1]{\protect\rule{.1in}{.1in}}
\providecommand{\U}[1]{\protect\rule{.1in}{.1in}}

\newtheorem{theorem}{Theorem}
{}

\newtheorem{definition}{Definition}

\newtheorem{lemma}{Lemma}
{}

\newtheorem{proposition}{Proposition}

\newenvironment{proof}[1][Proof]{\textbf{#1.} }{\ \rule{0.5em}{0.5em}}

\begin{document}

\title{On the Simplicity of the Eigenvalues of the Non-self-adjoint \ Mathieu-Hill
Operators }
\author{O. A. Veliev\\{\small Depart. of Math., Dogus University, Ac\i badem, Kadik\"{o}y, \ }\\{\small Istanbul, Turkey.}\ {\small e-mail: oveliev@dogus.edu.tr}}
\date{}
\maketitle

\begin{abstract}
We find conditions on the potential of the non-self-adjoint \ Mathieu-Hill
operator such that the all eigenvalues of the periodic, antiperiodic,
Dirichlet and Neumann boundary value problems are simple.

Key Words: Mathieu-Hill operator, Simple eigenvalues.

AMS Mathematics Subject Classification: 34L05, 34L20.

\end{abstract}

\section{Introduction and Preliminary Facts}

Let $P(q)$, $A(q),$ $D(q)$, $N(q)$ be the operators in $L_{2}[0,\pi]$
associated with the equation%

\begin{equation}
-y^{^{\prime\prime}}(x)+q(x)y(x)=\lambda y(x)
\end{equation}
and the periodic
\begin{equation}
y(\pi)=y(0),\text{ }y^{^{\prime}}(\pi)=y^{^{\prime}}(0),
\end{equation}
antiperiodic
\begin{equation}
y(\pi)=-y(0),\text{ }y^{^{\prime}}(\pi)=-y^{^{\prime}}(0),
\end{equation}
Dirichlet
\begin{equation}
y(\pi)=y(0)=0,
\end{equation}
Neumann
\begin{equation}
y^{^{\prime}}(\pi)=y^{^{\prime}}(0)=0
\end{equation}
boundary conditions respectively.

\ It is well known that the spectra of the operators $P(q)$ and $A(q)$ consist
of the eigenvalues $\lambda_{2n}$ and $\lambda_{2n+1}$, called as periodic and
antiperiodic eigenvalues, that are the roots of
\begin{equation}
F(\lambda)=2\text{ }\And\text{ }F(\lambda)=-2,
\end{equation}
where $n=0,1,...,$ $F(\lambda)=:\varphi^{^{\prime}}(\pi,\lambda)+\theta
(\pi,\lambda)$ is the Hill discriminant and $\varphi(x,\lambda),$
$\theta(x,\lambda)$ are the solutions of the equation (1) satisfying the
initial conditions
\begin{equation}
\theta(0,\lambda)=\varphi^{^{\prime}}(0,\lambda)=1,\quad\theta^{^{\prime}%
}(0,\lambda)=\varphi(0,\lambda)=0.
\end{equation}
The eigenvalues of the operators $D(q)$ and $N(q)$, called as Dirichlet and
Neumann eigenvalues, are the roots of
\begin{equation}
\varphi(\pi,\lambda)=0\text{ }\And\text{ \ }\theta^{^{\prime}}(\pi,\lambda)=0
\end{equation}
respectively. The spectrum of the operator $L(q)$ associated with (1) and the
boundary conditions
\begin{equation}
y(2\pi)=y(0),\text{ }y^{^{\prime}}(2\pi)=y^{^{\prime}}(0)
\end{equation}
is the union of the periodic and antiperiodic eigenvalues. In other words, the
spectrum of $L(q)$ consist of the eigenvalues $\lambda_{n}$ for $n=0,1,....$
that are the roots of the equation
\begin{equation}
(F(\lambda)-2)(F(\lambda)+2)=0.
\end{equation}

The operators $P(q)$, $A(q),$ $D(q)$ and $N(q)$ are denoted respectively by
$P(a,b)$, $A(a,b),$ $D(a,b)$ and $N(a,b)$ if%

\begin{equation}
q(x)=ae^{-i2x}+be^{i2x},
\end{equation}
where $a$ and $b$ are complex numbers. If $b=a$ then, for simplicity of the
notations, these operators are redenoted by $P(a)$, $A(a),$ $D(a)$ and $N(a).$
The eigenvalues of $P(a)$ and $A(a)$ are denoted by $\lambda_{2n}(a)$ and
$\lambda_{2n+1}(a)$ for $n=0,1,....$

We use the following two classical theorems (see p.8-9 of [8] and p.34-35 of [6]).

\begin{theorem}
If $q(x)$ is an even function, then $\varphi(x,\lambda)$ is an odd function
and $\theta(x,\lambda)$ is an even function. Periodic solutions are either
$\varphi(x,\lambda)$ or $\theta(x,\lambda)$ unless all solutions are periodic
(with period $\pi$ or $2\pi$). Moreover, the following equality holds
\begin{equation}
\varphi^{^{\prime}}(\pi,\lambda)=\theta(\pi,\lambda).
\end{equation}

\end{theorem}

\begin{theorem}
For all $n$ and for any nonzero $a$ the geometric multiplicity of the
eigenvalue $\lambda_{n}(a)$ of the operators $P(a)$ and $A(a)$ is $1$ ( that
is, there exists one eigenfunction corresponding to $\lambda_{n}(a)$) and the
corresponding eigenfunction is either $\varphi(x,\lambda_{n}(a))$ or
$\theta(x,\lambda_{n}(a)),$ where, for simplicity of the notations, the
solutions of the equation%
\begin{equation}
-y^{^{\prime\prime}}(x)+(2a\cos2x)y(x)=\lambda y(x)
\end{equation}
satisfying (7) are denoted also by $\varphi(x,\lambda)$ and $\theta
(x,\lambda)$.
\end{theorem}

In [8, 6] these theorems were proved for the real-valued potentials. However,
the proofs pass through for the complex-valued potentials without any change.

The spectrum of \ $P(a),$ $A(a),$ $D(a)$, $N(a)$ for $a=0$ are
\[
\{(2k)^{2}:k=0,1,...\},\{(2k+1)^{2}:k=0,1,...\},\{k^{2}:k=1,2,...\},\{k^{2}%
:k=0,1,...\}
\]
respectively. All eigenvalues of $P(0),$ except $0,$ and $A(0)$ are double,
while the eigenvalues of $D(0)$ and $N(0)$ are simple.

We use also the following \ result of [11].

\begin{theorem}
If $ab=cd$, then the Hill discriminants $F(\lambda,a,b)$ and $F(\lambda,c,d)$
(see (6)) for the operators $P(a,b)$ and $P(c,d)$ are the same.
\end{theorem}

By Theorem 2 the geometric multiplicity of the eigenvalues of $P(a)$ and
$A(a)$ for any nonzero complex number $a$ is $1.$ However, in the
non-self-adjoint case $a\in\mathbb{C}\backslash\mathbb{R},$ the multiplicity
(algebraic multiplicity) of these eigenvalues, in general, is not equal to
their geometric multiplicity, since the operators $P(a)$ and $A(a)$ may have
associated functions (generalized eigenfunctions). Thus in the non-self-
adjoint case the multiplicity (algebraic multiplicity) of the eigenvalues may
be any finite number when the geometric multiplicity is $1$ or $2.$ Therefore
the investigation of the multiplicity of the eigenvalues for complex-valued
potential is more complicated.

In this paper we find the conditions on $a$ such that the all eigenvalues of
the operators $P(a)$, $A(a),$ $D(a)$ and $N(a)$ are simple, namely we prove
the following

\begin{theorem}
(Main results for the operators $P(a)$, $A(a),$ $D(a)$ and $N(a)$):

$(a)$ If $0<\left\vert a\right\vert \leq\frac{8}{\sqrt{6}},$ then the all
eigenvalues of the operators $A(a)$ and $D(a)$ are simple.

$(b)$ If $0<\left\vert a\right\vert \leq\frac{4}{3},$ then the all eigenvalues
of the operators $P(a)$ and $N(a)$ are simple.
\end{theorem}

This theorem with Theorem 3 implies

\begin{theorem}
(Main results for the operators $A(a,b)$ and $P(a,b)$):

$(a)$ If $0<\left\vert ab\right\vert \leq\frac{64}{6},$ then the all
eigenvalues of the operator $A(a,b)$ are simple.

$(b)$ If $0<\left\vert ab\right\vert \leq\frac{16}{9},$ then the all
eigenvalues of the operator $P(a,b)$ are simple.
\end{theorem}

Note that there are a lot of papers about the asymptotic analyses and about
the basis property of the root functions of the operators $P(a,b)$ and
$A(a,b)$ (see [1-5, 7, 10] and the references in them). We do not discuss
those papers, since in this paper we consider the another aspects of these
operators and use only Theorems 1-3.

\section{On the Even Potentials}

In this section we analyze, in general, the even potentials. In the paper [9]
the following statements about the connections of the spectra of the operators
$P(q)$, $A(q),$ $D(q)$ and $N(q),$ where $q$ is an even potential, were proved.

\textit{Lemma 1 of [9]. If }$q$\textit{ is an even potential and }$\lambda
$\textit{ is an eigenvalue of both operators }$D(q)$ and $N(q),$\textit{ then
}%
\begin{equation}
F(\lambda)=\pm2,\text{ }\frac{dF}{d\lambda}=0,
\end{equation}
\textit{that is, }$\lambda$\textit{ is a multiple eigenvalue of }$L(q).$

\textit{Proposition 1 of [9]. Let }$q$\textit{ be an even potential.}
\textit{Then }$\lambda$\textit{ is an eigenvalue of }$L(q)$ \textit{if and
only if }$\lambda$\textit{ is an eigenvalue of }$D(q)$ \textit{or} $N(q).$

First using (12) and the Wronskian equality
\begin{equation}
\theta(\pi,\lambda)\varphi^{^{\prime}}(\pi,\lambda)-\varphi(\pi,\lambda
)\theta^{^{\prime}}(\pi,\lambda)=1
\end{equation}
we prove the following improvements of these statements.

\begin{theorem}
Let $q$ be an even complex-valued function. A complex number $\lambda$ is both
a Neumann and Dirichled eigenvalue if and only if it is an eigenvalue of the
operator $L(q)$ with geometric multiplicity $2.$
\end{theorem}

\begin{proof}
Suppose $\lambda$ is both a Neumann and Dirichled eigenvalue, that is, both
equality in (8) hold. On the other hand, it follows from (12), (8) and (15)
that
\begin{equation}
\theta(\pi,\lambda)=\varphi^{^{\prime}}(\pi,\lambda)=\pm1.
\end{equation}
Now using (8), (16) and (7) one can easily verify that both $\theta
(x,\lambda)$ and $\varphi(x,\lambda)$ satisfy either periodic or anti-periodic
boundary condition, that is, $\lambda$ is an eigenvalue of the operator $L(q)$
with geometric multiplicity $2.$

Conversely, if $\lambda$ is an eigenvalue of $L(q)$ with geometric
multiplicity $2,$ then both $\theta(x,\lambda)$ and $\varphi(x,\lambda)$
satisfy either periodic or anti-periodic boundary condition. Therefore by (7)
the equalities in (8) hold, that is, $\lambda$ is both Neumann and Dirichled eigenvalue.
\end{proof}

\begin{theorem}
Let $q$ be an even complex-valued function. A complex number $\lambda$\ is an
eigenvalue of multiplicity $s$ of the operator $L(q)$ if and only if it is an
eigenvalue of multiplicities $u$ and $v$ of the operators $D(q)$ and $N(q)$
respectively, where $u+v=s$ and $u=0$ ($v=0$) means that $\lambda$ is not an
eigenvalue of $D(q)$ ($N(q)).$
\end{theorem}

\begin{proof}
It is well-known and clear that $\lambda_{0}$\ is an eigenvalue of
multiplicities $u$, $v$ and $s$ of the operator $D(q)$, $N(q)$ and $L(q)$
respectively if and only if
\begin{equation}
\varphi(\pi,\lambda)=(\lambda_{0}-\lambda)^{u}f(\lambda)\text{, \ }%
\theta^{^{\prime}}(\pi,\lambda)=(\lambda_{0}-\lambda)^{v}g(\lambda)
\end{equation}
and
\begin{equation}
(F(\lambda)-2)(F(\lambda)+2)=(\lambda_{0}-\lambda)^{s}h(\lambda)\text{, \ }%
\end{equation}
where $f(\lambda_{0})\neq0,$ $g(\lambda_{0})\neq0$ and $h(\lambda_{0})\neq0.$
On the other hand by (12) and (15) we have
\begin{equation}
(F(\lambda)-2)(F(\lambda)+2)=4\theta^{2}(\pi,\lambda)-4=4(\theta(\pi
,\lambda)\varphi^{^{\prime}}(\pi,\lambda)-1)=4\varphi(\pi,\lambda
)\theta^{^{\prime}}(\pi,\lambda).
\end{equation}
Thus the proof of the theorem follows from (17)-(19)
\end{proof}

To analyze the periodic and antiperiodic eigenvalues in detail let us
introduce the following notations and definitions.

\begin{definition}
Let $\sigma(T)$ denotes the spectrum of the operator $T.$ A number $\lambda$
is called $PDN(q)$ (periodic, Dirichled and Neumann) eigenvalue if $\lambda
\in\sigma(P(q))\cap\sigma(D(q))\cap\sigma(N(q)).$ A number $\lambda\in
\sigma(P(q))\cap\sigma(D(q))$ is called $PD(q)$ (periodic and Dirichled)
eigenvalue if it is not $PDN(q)$ eigenvalue. A number $\lambda\in
\sigma(P(q))\cap\sigma(N(q))$ is called $PN(q)$ (periodic and Neumann)
eigenvalue if it is not $PDN(q)$ eigenvalue. Everywhere replacing $P(q)$ by
$A(q)$ we get the definition of $ADN(q),$ $AD(q)$ and $AN(q)$ eigenvalues.
\end{definition}

Using Theorems 6, 7, Definition 1 and the equality $\sigma(P(q))\cap
\sigma(A(q))=\emptyset$ we obtain

\begin{theorem}
Let $q$ be an even complex-valued function. Then

$(a)$ The spectrum of $P(q)$ is the union of the following three pairwise
disjoint sets: $\{PDN(q)$ eigenvalues$\},\{PD(q)$ eigenvalues$\}$ and
$\{PN(q)$ eigenvalues$\}.$

$(b)$ A complex number $\lambda$\ is an eigenvalue of geometric multiplicity
$2$ of the operator $P(q)$ if and only if it is $PDN(q)$ eigenvalue.

$(c)$ A complex number $\lambda$\ is an eigenvalue of geometric multiplicity
$1$ of the operator $P(q)$ if and only if it is either $PD(q)$ or $PN(q)$ eigenvalue.

The theorem continues to hold if $P(q),$ $PDN(q),PD(q)$ and $PN(q)$ are
replaced by $A(q),$ $ADN(q),$ $AD(q)$ and $AN(q)$ respectively.
\end{theorem}

Now we prove the main theorem of this section.

\begin{theorem}
Let $q$ be an even complex-valued function and $\lambda$ be an eigenvalue of
geometric multiplicity $1$ of the operator $P(q).$ Then the number $\lambda
$\ is an eigenvalue of multiplicity $s$ of $P(q)$ if and only if it is an
eigenvalue of multiplicity $s$ either of the operator $D(q)$ (first case) or
of the operator $N(q)$ (second case). In the first case the system of the root
functions of the operators $P(q)$ and $D(q)$ consists of the same
eigenfunction $\varphi(x,\lambda)$ and associated functions
\begin{equation}
\frac{\partial\varphi(x,\lambda)}{\partial\lambda},\frac{1}{2!}\frac
{\partial^{2}\varphi(x,\lambda)}{\partial\lambda^{2}},...,\frac{1}%
{(s-1)!}\frac{\partial^{s-1}\varphi(x,\lambda)}{\partial\lambda^{s-1}}.
\end{equation}

In the second case the system of the root function of the operators $P(q)$ and
$N(q)$ consists of the same eigenfunction $\theta(x,\lambda)$ and associated
functions
\begin{equation}
\frac{\partial\theta(x,\lambda)}{\partial\lambda},\frac{1}{2!}\frac
{\partial^{2}\theta(x,\lambda)}{\partial\lambda^{2}},...,\frac{1}{(s-1)!}%
\frac{\partial^{s-1}\theta(x,\lambda)}{\partial\lambda^{s-1}}.
\end{equation}

The theorem continues to hold if $P(q)$ is replaced by $A(q).$
\end{theorem}

\begin{proof}
Let $\lambda$\ be an eigenvalue of geometric multiplicity $1$ and multiplicity
$s$ of the operator $P(q).$ By Theorem 1 there are two cases.

Case 1. The corresponding eigenfunction is $\varphi(x,\lambda).$

Case 2. The corresponding eigenfunction is $\theta(x,\lambda).$

We consider Case 1. In the same way one can consider Case 2. In Case 1,
$\theta(x,\lambda)$ is not a periodic solution, that is, it does not satisfy
the periodic boundary condition (2). On the other hand, the first equality of
(6) with (12) and (7) implies that
\begin{equation}
\theta(\pi,\lambda)=1=\theta(0,\lambda),
\end{equation}
that is, $\theta(x,\lambda)$ satisfies the first equality in (2). Therefore
$\theta(x,\lambda)$ does not satisfies the second equality of (2), that is,%
\begin{equation}
\theta^{^{\prime}}(\pi,\lambda)\neq0.
\end{equation}
This inequality means that $v=0,$ where $v$ is defined in Theorem 7.
Therefore, by Theorem 7 we have $u=s,$ that is, $\lambda$\ is an eigenvalue of
multiplicity $s$ of the operator $D(q).$

Now suppose that $\lambda$\ is an eigenvalue of multiplicity $s$ of $D(q).$
Then by (8) and (7)
\begin{equation}
\varphi(\pi,\lambda)=0=\varphi(0,\lambda).
\end{equation}
On the other hand, using the first equality of (6), (12) and (7) we get
\begin{equation}
\varphi^{^{\prime}}(\pi,\lambda)=1=\varphi^{^{\prime}}(0,\lambda).
\end{equation}
Therefore $\varphi(x,\lambda)$\ is an eigenfunction of $P(q)$ corresponding to
the eigenvalue $\lambda.$ Then, by Theorem 1, $\theta(x,\lambda)$ is not a
periodic solution. This, as we noted above, implies (23) and the equality
$u=s.$ Thus, by Theorem 7, $\lambda$\ is an eigenvalue of multiplicity $s$ of
$P(q).$

If $\lambda$ is an eigenvalue of multiplicity $s$ of the operators $P(q)$ and
$D(q),$ then
\begin{equation}
F(\lambda)=2,\text{ }\frac{dF}{d\lambda}=0,\text{ }\frac{d^{2}F}{d\lambda^{2}%
}=0,...,\frac{d^{s-1}F}{d\lambda^{s-1}}=0
\end{equation}
and
\begin{equation}
\varphi(\pi,\lambda)=0,\frac{d\varphi(\pi,\lambda)}{d\lambda}=0,\frac
{d^{2}\varphi(\pi,\lambda)}{d\lambda^{2}}=0,...,\frac{d^{s-1}\varphi
(\pi,\lambda)}{d\lambda^{s-1}}=0.
\end{equation}
Since $\varphi(0,\lambda)=0$ and $\varphi^{^{\prime}}(0,\lambda)=1$ for all
$\lambda,$ we have
\begin{equation}
\varphi(0,\lambda)=0,\frac{d\varphi(0,\lambda)}{d\lambda}=0,\frac{d^{2}%
\varphi(0,\lambda)}{d\lambda^{2}}=0,...,\frac{d^{s-1}\varphi(0,\lambda
)}{d\lambda^{s-1}}=0
\end{equation}
and
\begin{equation}
\varphi^{^{\prime}}(0,\lambda)=1,\frac{d\varphi^{^{\prime}}(0,\lambda
)}{d\lambda}=0,\frac{d^{2}\varphi^{^{\prime}}(0,\lambda)}{d\lambda^{2}%
}=0,...,\frac{d^{s-1}\varphi^{^{\prime}}(0,\lambda)}{d\lambda^{s-1}}=0.
\end{equation}
Moreover, using (26) and (12) we obtain
\begin{equation}
\varphi^{^{\prime}}(\pi,\lambda)=1,\frac{d\varphi^{^{\prime}}(\pi,\lambda
)}{d\lambda}=0,\frac{d^{2}\varphi^{^{\prime}}(\pi,\lambda)}{d\lambda^{2}%
}=0,...,\frac{d^{s-1}\varphi^{^{\prime}}(\pi,\lambda)}{d\lambda^{s-1}}=0.
\end{equation}
Thus, by (27)-(30), $\varphi(x,\lambda)$ and the functions in (20) satisfy
both the periodic and Dirichlet boundary conditions. On the other hand,
differentiating $s-1$ times, with respect to $\lambda,$ the equation
\begin{equation}
-\varphi^{^{\prime\prime}}(x,\lambda)+q(x)\varphi(x,\lambda)=\lambda
\varphi(x,\lambda)
\end{equation}
we obtain
\[
-(\frac{1}{k!}\frac{\partial^{k}\varphi(x,\lambda)}{\partial\lambda^{k}%
})^{^{\prime\prime}}+(q(x)-\lambda)\frac{1}{k!}\frac{\partial^{k}%
\varphi(x,\lambda)}{\partial\lambda^{k}}=\frac{1}{(k-1)!}\frac{\partial
^{k-1}\varphi(x,\lambda)}{\partial\lambda^{k-1}}%
\]
\ for $k=1,2,...,(s-1).$ Therefore $\varphi(x,\lambda)$ and the functions in
(20) are the root functions of the operators $P(q)$ and $D(q).$ Thus the first
case is proved. In the same way we prove the second case. The proof of this
results for $A(q)$ are similar
\end{proof}

\section{Main Results}

In this section we consider the operators $P(a)$, $A(a),$ $D(a)$ and $N(a)$
with potential%
\begin{equation}
q(x)=2a\cos2x,
\end{equation}
where $a$ is a nonzero complex number. By Theorem 2 the geometric multiplicity
of the eigenvalues of $P(a)$ and $A(a)$ is $1.$ Therefore it follows from
Theorem 8 that%
\begin{equation}
\sigma(P(a))=\{PD(a)\text{ eigenvalues}\}\cup\{PN(a)\text{ eigenvalues}\},
\end{equation}%
\begin{equation}
\sigma(A(a))=\{AD(a)\text{ eigenvalues}\}\cup\{AN(a)\text{ eigenvalues}\},
\end{equation}
where $PD(q),$ $PN(q),$ $AD(q)$ and $AN(q)$ (see Definition 1) are denoted by
$PD(a),$ $PD(a),$ $PD(a)$ and $PD(a)$ when the potential $q$ is defined by
(32). Moreover, Theorem 7, Theorem 2 and Theorem 9 yield the equalities
\begin{equation}
\sigma(D(a))=\{PD(a)\text{ eigenvalues}\}\cup\{AD(a)\text{ eigenvalues}\},
\end{equation}%
\begin{equation}
\sigma(N(a))=\{PN(a)\text{ eigenvalues}\}\cup\{AN(a)\text{ eigenvalues}\}
\end{equation}
and the following theorem.

\begin{theorem}
For any $a\neq0$ the eigenvalue $\lambda$ of the operator $P(a)$ or $A(a)$ is
multiple if and only if it is a multiple \ eigenvalue either of $D(a)$ or
$N(a)$. Moreover, the operators $P(a)$, $A(a),$ $D(a)$ and $N(a)$ have
associated functions corresponding to any multiple eigenvalues.
\end{theorem}

Clearly, the eigenfunctions corresponding to $PN(a)$ eigenvalues, $PD(a)$
eigenvalues, $AD(a)$ eigenvalues and $AN(a)$ eigenvalues have the forms
\begin{equation}
\Psi_{PN}(x)=\frac{a_{0}}{\sqrt{2}}+\sum_{k=1}^{\infty}a_{k}\cos2kx,
\end{equation}%
\begin{equation}
\Psi_{PD}(x)=\sum_{k=1}^{\infty}b_{k}\sin2kx,
\end{equation}%
\begin{equation}
\Psi_{AD}(x)=\sum_{k=1}^{\infty}c_{k}\sin(2k-1)x,
\end{equation}
and%
\begin{equation}
\Psi_{AN}(x)=\sum_{k=1}^{\infty}d_{k}\cos(2k-1)x
\end{equation}
respectively. For simplicity of the calculating we normalize these
eigenfunctions as follows
\begin{equation}
\sum_{k=0}^{\infty}\left\vert a_{k}\right\vert ^{2}=1,\text{ }\sum
_{k=1}^{\infty}\left\vert b_{k}\right\vert ^{2}=1,\text{ }\sum_{k=1}^{\infty
}\left\vert c_{k}\right\vert ^{2}=1,\text{ }\sum_{k=1}^{\infty}\left\vert
d_{k}\right\vert ^{2}=1.
\end{equation}
Substituting the functions (37)-(40) into (13) we obtain the following
equalities
\begin{equation}
\lambda a_{0}=\sqrt{2}aa_{1},\text{ }(\lambda-4)a_{1}=a\sqrt{2}a_{0}%
+aa_{2},\text{ }(\lambda-(2k)^{2})a_{k},=aa_{k-1}+aa_{k+1},
\end{equation}%
\begin{equation}
(\lambda-4)b_{1}=ab_{2},\text{ }(\lambda-(2k)^{2})b_{k},=ab_{k-1}+ab_{k+1},
\end{equation}%
\begin{equation}
(\lambda-1)c_{1}=ac_{1}+ac_{2},\text{ }(\lambda-(2k-1)^{2})c_{k}%
,=ac_{k-1}+ac_{k+1},
\end{equation}%
\begin{equation}
(\lambda-1)d_{1}=-ad_{1}+ad_{2},\text{ }(\lambda-(2k-1)^{2})d_{k}%
,=ad_{k-1}+ad_{k+1}%
\end{equation}
for $k=2,3,...$. Here $a_{k},$ $b_{k},$ $c_{k},$ $d_{k}$ depend on $\lambda$
and $a_{0},b_{1},c_{1}$, $d_{1}$ are nonzero constants (see [6] p. 34-35).

By Theorem 10, if the eigenvalue $\lambda$ corresponding to one of the
eigenfunctions (37)-(40), denoted by $\Psi(x)$, is multiple then there exists
associated function $\Phi$ satisfying
\begin{equation}
-(\Phi(x,\lambda))^{^{\prime\prime}}+(q(x)-\lambda)\Phi(x,\lambda)=\Psi(x).
\end{equation}
Since the boundary conditions (2)-(5) are self-adjoint $\overline{\lambda}$
and $\overline{\Psi(x)}$ are eigenvalue and eigenfunction of the adjoint
operator. Therefore multiplying both sides of (46) by $\overline{\Psi}$ we get
$(\Psi,\overline{\Psi})=0,$ where $(.,.)$ is the inner product in $L_{2}%
[0,\pi].$ Thus, if the eigenvalues corresponding to the eigenfunctions
(37)-(40), are multiple, then we have
\begin{equation}
\sum_{k=0}^{\infty}a_{k}^{2}=0,\text{ }\sum_{k=1}^{\infty}b_{k}^{2}=0,\text{
}\sum_{k=1}^{\infty}c_{k}^{2}=0,\text{ }\sum_{k=1}^{\infty}d_{k}^{2}=0.
\end{equation}

To prove the simplicity of the eigenvalue $\lambda$ corresponding, say, to
(40) we show that there is not a sequence $\{d_{k}\}$ satisfying the above $3$
equalities: (45), (41) and (47), since these equalities hold if $\lambda$ is a
multiple eigenvalue. For this we use following proposition which readily
follows from (41) and (47).

\begin{proposition}
If there exists $n\in\mathbb{N=}:\{1,2,...,\}$ such that
\begin{equation}
\left\vert d_{n}(\lambda)\right\vert ^{2}>\frac{1}{2},
\end{equation}
then $\lambda$ is a simple $AN(a)$ eigenvalue, where $a\neq0$. The statement
continues to hold for $AD(a)$, $PD(a)$ and $PN(a)$ eigenvalues if $d_{n}$ is
replaced by $c_{n}$, $b_{n}$ and $a_{n}$ respectively.
\end{proposition}

To apply the Proposition 1, we use following lemmas.

\begin{lemma}
Suppose that $\lambda$ is a multiple $AN(a)$ eigenvalue corresponding to the
eigenfunction (40), where $a\neq0$. Then

$(a)$ For all $k\in\mathbb{N}$, $m\in\mathbb{N}$, $k\neq m$ the following
inequalities hold%
\begin{align}
\left\vert d_{k}\right\vert ^{2}  &  \leq\frac{1}{2},\\
\text{ }\left\vert d_{k}\pm d_{m}\right\vert ^{2}  &  \leq1,
\end{align}%
\begin{equation}
\left\vert d_{k}\right\vert ^{2}\leq\frac{\left\vert a\right\vert ^{2}%
}{\left\vert \lambda-(2k-1)^{2}\right\vert ^{2}}.
\end{equation}

$(b)$ If $\operatorname{Re}\lambda<(2p-1)^{2}-2\left\vert a\right\vert $ for
some $p\in\mathbb{N}$, then $\left\vert d_{k-1}\right\vert >\left\vert
d_{k}\right\vert >0$ and%
\begin{equation}
\left\vert d_{k+s}\right\vert <\frac{\left\vert 2a\right\vert ^{s+1}\left\vert
d_{k-1}\right\vert }{\left\vert \lambda-(2k-1)^{2}\right\vert \left\vert
\lambda-(2(k+1)-1)^{2}\right\vert ...\left\vert \lambda-(2(k+s)-1)^{2}%
\right\vert }%
\end{equation}
\ 

\ \ \ \ for all $k>p$ and $s=0,1,....$

$(c)$ Let $I\subset\mathbb{N}$ and $d(\lambda,I)=:\min_{k\in I}\left\vert
\lambda-(2k-1)^{2}\right\vert \neq0.$ Then
\begin{equation}
\sum_{k\in I}\left\vert d_{k}\right\vert ^{2}\leq\frac{4\left\vert
a\right\vert ^{2}}{(d(\lambda,I))^{2}}.
\end{equation}

$(d)$ If $\lambda$ is a multiple $AD(a)$ eigenvalue corresponding to the
eigenfunction (39), then the inequalities (49)-(53) continues to hold if
$d_{j}$ is replaced by $c_{j}.$
\end{lemma}

\begin{proof}
$(a)$ If (49) does not hold for some $k,$ then by Proposition 1 $\lambda$ is a
simple eigenvalue that contradicts the assumption of the lemma.

Using the last equalities of (47) and (41), we obtain
\[
\left\vert (d_{k}\pm d_{m})^{2}\right\vert =\left\vert -\sum_{n\neq k,m}%
d_{n}^{2}\pm2d_{k}d_{m}\right\vert \leq\sum_{n\neq k,m}\left\vert
d_{n}\right\vert ^{2}+\left\vert d_{k}\right\vert ^{2}+\left\vert
d_{m}\right\vert ^{2}=1,
\]
that is, (50) holds. Now (51) follows from (45) and (50).

$(b)$ Suppose that $\left\vert d_{k}\right\vert \geq\left\vert d_{k-1}%
\right\vert $ for some $k>p>0.$ By (45)
\[
\left\vert \lambda-(2k-1)^{2}\right\vert \left\vert d_{k}\right\vert
\leq\left\vert a\right\vert \left\vert d_{k-1}\right\vert +\left\vert
a\right\vert \left\vert d_{k+1}\right\vert .
\]
On the other hand, using the condition on $\lambda$ we get $\left\vert
\lambda-(2k-1)^{2}\right\vert >2\left\vert a\right\vert .$ Therefore
\[
\left\vert d_{k+1}\right\vert \geq2\left\vert d_{k}\right\vert -\left\vert
d_{k-1}\right\vert \geq\left\vert d_{k}\right\vert .
\]
Repeating this process $s$ times we obtain $\left\vert d_{k+s}\right\vert
\geq\left\vert d_{k+s-1}\right\vert $ for all $s\in\mathbb{N}$. It means that
$\{\left\vert d_{k+s}\right\vert $ $:s\in\mathbb{N\}}$ is a nondecreasing
sequence. \ On the other hand, $\left\vert d_{k}\right\vert +\left\vert
d_{k+1}\right\vert \neq0,$ since if both $d_{k}$ and $d_{k+1}$ are zero, then
using (45) we obtain that $d_{j}=0$ for all $j\in\mathbb{N},$ that is, the
solutions \ (40) is identically zero. Therefore $d_{k}$ does not converge to
zero being the Fourier coefficient of the square integrable function
$\Psi_{AN}(x)$. This contradiction shows that $\{\left\vert d_{k+s}\right\vert
$ $:s\in\mathbb{N\}}$ is a decreasing sequence. Thus $\left\vert
d_{k}\right\vert >0$ for all $k>p.$

Now let us prove (52). Using (45) and the inequality $\left\vert
d_{k-1}\right\vert >\left\vert d_{k}\right\vert >0,$ we get
\begin{equation}
\left\vert d_{k+s}\right\vert <\frac{\left\vert 2a\right\vert \left\vert
d_{k+s-1}\right\vert }{\left\vert \lambda-(2(k+s)-1)^{2}\right\vert }%
\end{equation}
for all $s=0,1,....$ Iterating (54) $s$ times we obtain (52).

$(c)$ By (45) we have
\[
\sum_{k\in I}\left\vert d_{k}\right\vert ^{2}\leq\sum_{k\in I}\frac{\left\vert
a\right\vert ^{2}(\left\vert d_{k-1}\right\vert +\left\vert d_{k+1}\right\vert
)^{2}}{(d(\lambda,I))^{2}}\leq\sum_{k\in I}\frac{2\left\vert a\right\vert
^{2}(\left\vert d_{k-1}\right\vert ^{2}+\left\vert d_{k+1}\right\vert ^{2}%
)}{(d(\lambda,I))^{2}}.
\]
Note that in case $k=1$ instead of $d_{k-1}$ we take $d_{1}$ (see the first
equality of (45)). Now (53) follows from (41).

$(d)$ Everywhere replacing $d_{k}$ by $c_{k}$ we get the proof of the last statement
\end{proof}

In the similar way we prove the following lemma for $P(a).$

\begin{lemma}
Suppose that $\lambda$ is a multiple $PD(a)$ eigenvalue corresponding to the
eigenfunction (38), where $a\neq0$. Then

$(a)$ For all $k\in\mathbb{N}$, $m\in\mathbb{N}$, $n\in\mathbb{N}$, $n\neq m$
the following inequalities hold%
\begin{equation}
\left\vert b_{m}\right\vert ^{2}\leq\frac{1}{2},\text{ }\left\vert b_{n}\pm
b_{m}\right\vert ^{2}\leq1,\text{ }\left\vert b_{k}\right\vert ^{2}\leq
\frac{\left\vert a\right\vert ^{2}}{\left\vert \lambda-(2k)^{2}\right\vert
^{2}}.
\end{equation}

$(b)$ If $\operatorname{Re}\lambda<(2p)^{2}-2\left\vert a\right\vert $ for
some $p\in\mathbb{N}$, then $\left\vert b_{k-1}\right\vert >\left\vert
b_{k}\right\vert >0$ and%
\begin{equation}
\text{ }\left\vert b_{k+s}\right\vert <\frac{\left\vert 2a\right\vert
^{s+1}\left\vert b_{k-1}\right\vert }{\left\vert \lambda-(2k)^{2}\right\vert
\left\vert \lambda-(2(k+1))^{2}\right\vert ...\left\vert \lambda
-(2(k+s))^{2}\right\vert }%
\end{equation}

for all $k>p$ and $s=0,1,...$

$(c)$ Let $I\subset\mathbb{N}$ and $b(\lambda,I)=\min_{k\in I}\left\vert
\lambda-(2k)^{2}\right\vert \neq0.$ Then
\begin{equation}
\sum_{k\in I}\left\vert b_{k}\right\vert ^{2}\leq\frac{4\left\vert
a\right\vert ^{2}}{(b(\lambda,I))^{2}}.
\end{equation}

$(d)$ If $\lambda$ is a multiple $PN(a)$ eigenvalue corresponding to (37) then
the statements $(a)$ and $(b)$ continue to hold for $k>1$, $m\geq0$ and the
statement $(c)$ continues to hold for $I\subset\{0,1,...\}$ if $b_{j}$ is
replaced by $a_{j}.$
\end{lemma}

Introduce the notation $D_{n}=\{\lambda\in\mathbb{C}:\left\vert \lambda
-(2n-1)^{2}\right\vert \leq2\left\vert a\right\vert $ $\}.$

\begin{theorem}
$(a)$ All eigenvalues of the operator $A(a)$ lie on the unions of $D_{n}$ for
$n\in\mathbb{N}$.

$(b)$ If $4n-4>(1+\sqrt{2})\left\vert a\right\vert $, where $a\neq0,$ then the
eigenvalues of $A(a)$ lying in $D_{n}$ are simple.
\end{theorem}

\begin{proof}
By (34) if $\lambda$ is an eigenvalues of the operator $A(a),$ then the
corresponding eigenfunction is either $\Psi_{AN}(x)$ or $\Psi_{AD}(x)$ (see
(39) or (40)). Without loss of generality, we assume that the corresponding
eigenfunction is $\Psi_{AN}(x)$.

$(a)$ Since $d_{k}\rightarrow0$ as $k\rightarrow\infty,$ there exists
$n\in\mathbb{N}$ such that
\[
\left\vert d_{n}\right\vert =\max_{k\in\mathbb{N}}\left\vert d_{k}\right\vert
.
\]
Therefore $(a)$ follows from (45) for $k=n.$

$(b)$ Suppose that $\lambda\in D_{n\text{ }}$ is a multiple eigenvalue
corresponding to the eigenfunction $\Psi_{AN}(x).$ By definition of
$D_{n\text{ }}$ for $k\neq n$ we have
\[
\left\vert \lambda-(2k-1)^{2}\right\vert \geq\mid(2n-1)^{2}-(2k-1)^{2}%
\mid-\left\vert 2a\right\vert \geq\left\vert (2n-3)^{2}-(2n-1)^{2}\right\vert
-\left\vert 2a\right\vert .
\]
This together with the condition on $n$ and the definition of \ $d(\lambda,I)$
(see Lemma 1(c)) gives $d(\lambda,\mathbb{N}\backslash\{n\})>2\sqrt
{2}\left\vert a\right\vert .$ Thus, using (53) and (41) we get%
\[
\sum_{k\neq n}\left\vert d_{k}\right\vert ^{2}<\frac{1}{2}\text{ }\And\text{
}\left\vert d_{n}\right\vert ^{2}>\frac{1}{2}%
\]
which contradicts Proposition 1.
\end{proof}

Instead of Lemma 1 using Lemma 2 in the same way we prove the following

\begin{theorem}
$(a)$ All $PD(a)$ eigenvalues lie in the unions of $B=:\{\lambda:\left\vert
\lambda-4\right\vert \leq\left\vert a\right\vert $ $\}$ and

$B_{n}=:\{\lambda:\left\vert \lambda-(2n)^{2}\right\vert \leq2\left\vert
a\right\vert $ $\}$ for $n=2,3,.....$ All $PN(a)$ eigenvalues lie in the
unions of $A_{0}=\{\lambda:\left\vert \lambda\right\vert \leq\sqrt
{2}\left\vert a\right\vert $ $\},$ $A_{1}=\{\lambda:\left\vert \lambda
-4\right\vert \leq(1+\sqrt{2})\left\vert a\right\vert $ $\}$ and $B_{n}$ for
$n=2,3,....$

$(b)$ If $4n-2>(1+\sqrt{2})\left\vert a\right\vert $ and $n>1,$ where
$a\neq0,$ then the eigenvalues of $P(a)$ lying in $B_{n}$ are simple.
\end{theorem}

Now we prove the main result for $A(a).$

\begin{theorem}
If $0<\left\vert a\right\vert \leq\frac{8}{\sqrt{6}},$ then the all
eigenvalues of the operator $A(a)$ are simple.
\end{theorem}

\begin{proof}
Since $8>\frac{8}{\sqrt{6}}(1+\sqrt{2}),$ by Theorem 11(b) the ball
$D_{n\text{ }}$ for $n>2$ does not contain the multiple eigenvalues of the
operator $A(a).$ Therefore we need to prove that the ball $D_{n\text{ }}$ for
$n=1,2$ also does not contain the multiple eigenvalues. Since the balls
$D_{1\text{ }}$ and $D_{2\text{ }}$ are contained in the half plane
$\{\lambda\in\mathbb{C}:\operatorname{Re}\lambda<16$ $\}$ we consider the
following two strips $\{\lambda\in\mathbb{C}:9<\operatorname{Re}\lambda<16$
$\}$, $\{\lambda\in\mathbb{C}:6<\operatorname{Re}\lambda\leq9$ $\}$ and half
plane $\{\lambda\in\mathbb{C}:\operatorname{Re}\lambda\leq6$ $\}$ separately.
We consider the $AN(a)$ eigenvalues, that is, the eigenvalues corresponding to
the eigenfunction (40). Consideration of the $AD(a)$ eigenvalues are the same.

To prove the simplicity of the eigenvalues lying in the above strips, we
assume that $\lambda$ is a multiple eigenvalue. Using Lemma 1 by direct
calculating (see Estimation 1 and Estimation 2 in Appendix) we show that (48)
for $n=2$ holds that contradicts Proposition 1.

Investigation the half plane $\operatorname{Re}\lambda\leq6$ is more
complicated. Here we use \ the first two equalities of (45)
\begin{equation}
(\lambda-1)d_{1}=-ad_{1}+ad_{2},\text{ }(\lambda-9)d_{2},=ad_{1}+ad_{3}.
\end{equation}
By direct calculating we get (see Estimation 3 and Estimation 4 in the
Appendix)
\begin{equation}
\sum_{k=3}^{\infty}\left\vert d_{k}\right\vert ^{2}<0.03\,\allowbreak
415,\text{ }\frac{\left\vert d_{3}\right\vert }{\left\vert d_{2}\right\vert
}<0.174\,32
\end{equation}
Then by (41) we have%
\begin{equation}
\left\vert d_{1}\right\vert ^{2}+\left\vert d_{2}\right\vert ^{2}%
>1-\varepsilon,
\end{equation}
where $\varepsilon=0.03\,\allowbreak415.$ On the other hand, by (49),
$\left\vert d_{1}\right\vert ^{2}\leq\frac{1}{2},$ $\left\vert d_{2}%
\right\vert ^{2}\leq\frac{1}{2}.$ These inequalities and (47) imply that%
\[
\left\vert d_{1}\right\vert ^{2}=\frac{1}{2}-\varepsilon_{1},\text{
}\left\vert d_{2}\right\vert ^{2}=\frac{1}{2}-\varepsilon_{2},\text{ }%
d_{2}^{2}=-\text{ }d_{1}^{2}+\varepsilon_{3},
\]
where $\varepsilon_{1}\geq0,$ $\varepsilon_{2}\geq0,$ $\varepsilon
_{1}+\varepsilon_{2}=\varepsilon,$ $\left\vert \varepsilon_{3}\right\vert
<0.03\,\allowbreak415.$ Now, one can easily see that
\[
(\frac{d_{2}}{d_{1}})^{2}=-1+\alpha,\frac{d_{2}}{d_{1}}=\pm(i+\delta),
\]
where $\left\vert \alpha\right\vert <\frac{0.03\,\allowbreak415}%
{0.5-0.03\,\allowbreak415}<0.074,$ $\left\vert \delta\right\vert <\frac{1}%
{2}\left\vert 0.074\right\vert +\frac{1}{7}\left\vert 0.074\right\vert
^{2}\allowbreak<0.0\,4$. Therefore we have
\begin{equation}
\frac{d_{2}}{d_{1}}-\frac{d_{1}}{d_{2}}=\pm\frac{(i+\delta)^{2}-1}{i+\delta
}=\pm\frac{2i(i+\delta)+\delta^{2}}{i+\delta}=\pm2i+\gamma,
\end{equation}
where $\left\vert \gamma\right\vert <\frac{(0.04)^{2}}{1-0.04}<0.002.$ On the
other hand, dividing the first equality of (58) by $d_{1}$ and the second by
$d_{2}$ and then subtracting second from the first and taking into account
(61) we get
\begin{equation}
\frac{8}{a}=\pm2i-1+\gamma-\frac{d_{3}}{d_{2}},
\end{equation}
where by assumption $\left\vert \frac{8}{a}\right\vert \geq\sqrt{6}.$
Therefore using the second estimation of (59) in (62) we get the contradiction

$2.\,\allowbreak449\,5<\sqrt{6}\leq\left\vert \frac{8}{a}\right\vert <\sqrt
{5}+0.174\,32+0.002<\allowbreak2.\,\allowbreak412\,5$
\end{proof}

In the same way we consider the simplicity of the eigenvalues of the operators
$P(a),$ $D(a)$ and $N(a).$ \ First let us investigate the eigenvalues of
$D(a).$ \ Since the eigenvalues of $D(a)$ is the union of $PD(a)$ and $AD(a)$
eigenvalues and the $AD(a)$ eigenvalues are investigated in Theorem 13, we
investigate the $PD(a)$ eigenvalue.

\begin{theorem}
If $0<\left\vert a\right\vert \leq5,$ then all $PD(a)$ eigenvalues are simple.
Moreover, if 

$0<\left\vert a\right\vert \leq\frac{8}{\sqrt{6}},$ then the all eigenvalues
of the operator $D(a)$ are simple.
\end{theorem}

\begin{proof}
The second statement follows from the first statement and Theorem 13.
Therefore we need to prove the first statement by using (43). Since
$14>5(1+\sqrt{2}),$ by Theorem 12,\ the $PD(a)$ eigenvalues lying in the ball
$B_{n\text{ }}$ for $n>3$ are simple.

If $\lambda\in B_{3},$ then $26\leq\operatorname{Re}\lambda\leq46$. Using
Lemma 2 and (41) we obtain the estimations (see Estimation 5 in Appendix)
\[
\sum_{k\neq3}\left\vert b_{k}\right\vert ^{2}<\frac{1}{2},\text{ }\left\vert
b_{3}\right\vert ^{2}>\frac{1}{2}%
\]
which, by Proposition 1, proves the simplicity of the $PD(a)$ eigenvalues
lying in $B_{3}$.

Now we need to prove that the balls $B$ and $B_{2\text{ }}$ does not contain
the multiple $PD(a)$ eigenvalues. Since these balls are contained in the strip
$\{\lambda\in\mathbb{C}:\operatorname{Re}\lambda\leq26$ $\}$ we consider the
following cases: $16<\operatorname{Re}\lambda\leq26,$ $12<\operatorname{Re}%
\lambda\leq16$ and $\operatorname{Re}\lambda\leq12.$

In the first two cases using Lemma 2 we get the inequality (see Estimation 6
and Estimation 7) obtained from (48) for $n=2$ by replacing $d_{n\text{ }}$
with $b_{n}$ which proves, by Proposition 1, the simplicity of the eigenvalues.

Now consider the third case $\operatorname{Re}\lambda\leq12.$ Using Lemma 2 we
obtain (see Estimation 8 and Estimation 9 in Appendix)%
\begin{equation}
\sum_{k=3}^{\infty}\left\vert b_{k}\right\vert ^{2}<\frac{1}{15},\text{ }%
\frac{\left\vert b_{3}\right\vert }{\left\vert b_{2}\right\vert }<0.213\,1
\end{equation}
The first inequality of (63) with (41) implies that
\begin{equation}
\left\vert b_{1}\right\vert ^{2}+\left\vert b_{2}\right\vert ^{2}>1-\beta,
\end{equation}
where $\beta<\allowbreak\frac{1}{15}.$ Instead of (60) using (64) and
repeating the proof of (61) we obtain
\begin{equation}
\frac{b_{2}}{b_{1}}-\frac{b_{1}}{b_{2}}=\frac{(i+\delta)^{2}-1}{i+\delta
}=\frac{2i(i+\delta)+\delta^{2}}{i+\delta}=\pm2i+\gamma_{1},
\end{equation}
where $\left\vert \gamma_{1}\right\vert <0.01.$ Now dividing the first
equality of (43) by $b_{1}$ and the second equality of (43) for $k=2$ by
$b_{2}$ and then subtracting second from the first and using (65) we get
\begin{equation}
\frac{12}{a}=\pm2i+\gamma_{1}-\frac{b_{3}}{b_{2}},
\end{equation}
where by assumption $\left\vert \frac{12}{a}\right\vert \geq2.4.$ Thus, using
(63) in (66) we get the contradiction

$2.\,\allowbreak4\leq\left\vert \frac{12}{a}\right\vert
<2+0.213\,1+0.01=\allowbreak2.\,\allowbreak223\,1$
\end{proof}

\begin{theorem}
If $0<\left\vert a\right\vert \leq\frac{4}{3},$ then the all eigenvalues of
the operators $P(a)$ and $N(a)$ are simple.
\end{theorem}

\begin{proof}
By Theorem 13 and Theorem 14 we need to prove that if $\left\vert a\right\vert
\leq\frac{4}{3},$ then all $PN(a)$ eigenvalues are simple. Since
$6>(1+\sqrt{2})\frac{4}{3},$ by Theorem 12,\ the $PN(a)$ eigenvalues lying in
the ball $B_{n\text{ }}$ for $n>1$ are simple.

Now we prove that the balls $A_{0\text{ }}$ and $A_{1\text{ }}$ does not
contain the multiple $PN(a)$ eigenvalues. Since these balls are contained in
$\{\lambda\in\mathbb{C}:\operatorname{Re}\lambda<8$ $\}$ we consider the
following cases:

Case 1: $3\leq\operatorname{Re}\lambda<8.$ Using (42) and Lemma 2 (see
Estimation 10 in Appendix) we obtain $\left\vert a_{1}\right\vert ^{2}%
>\frac{1}{2}$ which, by Proposition 1, proves the simplicity of the eigenvalues.

Case 2: $\operatorname{Re}\lambda<3.$ Using Lemma 2 we obtain ( see
Estimations 11 and 12 in Appendix)%
\begin{equation}
\sum_{k=2}^{\infty}\left\vert a_{k}\right\vert ^{2}<\frac{1}{58},\text{ }%
\frac{\left\vert a_{2}\right\vert }{\left\vert a_{1}\right\vert }<0.103\,01
\end{equation}
The first inequality of (67) with (41) implies that
\begin{equation}
\left\vert a_{0}\right\vert ^{2}+\left\vert a_{1}\right\vert ^{2}>1-\rho,
\end{equation}
where $\rho<\frac{1}{58}.$ Instead of (60) using (68) and repeating the proof
of (61) we obtain%
\begin{equation}
\frac{a_{1}}{a_{0}}-\frac{a_{0}}{a_{1}}=\pm2i+\gamma,
\end{equation}
where $\left\vert \gamma\right\vert <0.0006.$ Now dividing the first equality
of (42) by $a_{0}$ and the second by $a_{1}$ and then subtracting second from
the first and taking into account (69) we get
\begin{equation}
\frac{4}{a}=\pm2\sqrt{2}i+\sqrt{2}\gamma-\frac{a_{2}}{a_{1}},
\end{equation}
where by assumption $\left\vert \frac{4}{a}\right\vert \geq3.$ Therefore using
(67) we get the contradiction

$3\leq\left\vert \frac{4}{a}\right\vert <\sqrt{2}%
(2+0.0006)+0.103\,01=\allowbreak2.\,\allowbreak932\,3$
\end{proof}

\section{Appendix}

\textbf{Estimation 1}: Let $9<\operatorname{Re}\lambda<16$. \bigskip By (51)
we have
\[
\left\vert d_{1}\right\vert ^{2}\leq\frac{\left\vert a\right\vert ^{2}%
}{\left\vert \lambda-1\right\vert ^{2}}\leq\frac{\left\vert \frac{8}{\sqrt{6}%
}\right\vert ^{2}}{\left\vert 8\right\vert ^{2}}=\frac{1}{6},\text{
}\left\vert d_{3}\right\vert ^{2}\leq\frac{\left\vert a\right\vert ^{2}%
}{\left\vert \lambda-25\right\vert ^{2}}\leq\frac{\left\vert \frac{8}{\sqrt
{6}}\right\vert ^{2}}{\left\vert 9\right\vert ^{2}}=\frac{32}{243}.
\]
Since $d(\lambda,\{4,5,...\})<33$ using (53) we get%
\[
\sum_{k=4}^{\infty}\left\vert d_{k}\right\vert ^{2}<\frac{4\left\vert \frac
{8}{\sqrt{6}}\right\vert ^{2}}{\left\vert 33\right\vert ^{2}}=\frac{128}%
{3267}.
\]
These inequalities imply that
\[
\sum_{k\neq2}\left\vert d_{k}\right\vert ^{2}<\frac{128}{3267}+\frac{32}%
{243}+\frac{1}{6}=\frac{19\,849}{58\,806}<\frac{1}{2}.
\]

\textbf{Estimation 2.} Let $6<\operatorname{Re}\lambda\leq9.$\bigskip\ By
(51)
\[
\left\vert d_{1}\right\vert ^{2}\leq\frac{\left\vert \frac{8}{\sqrt{6}%
}\right\vert ^{2}}{\left\vert 5\right\vert ^{2}}=\frac{32}{75},\text{
}\left\vert d_{3}\right\vert ^{2}\leq\frac{\left\vert \frac{8}{\sqrt{6}%
}\right\vert ^{2}}{\left\vert 16\right\vert ^{2}}=\frac{1}{24}.
\]
Now using the obvious equality $d(\lambda,\{4,5,...\})\leq40$ and (53), we get%
\[
\sum_{k=4}^{\infty}\left\vert d_{k}\right\vert ^{2}\leq\frac{4\left\vert
\frac{8}{\sqrt{6}}\right\vert ^{2}}{\left\vert 40\right\vert ^{2}}=\frac
{2}{75},\text{ }\sum_{k\neq2}\left\vert d_{k}\right\vert ^{2}\leq\frac{32}%
{75}+\frac{1}{24}+\frac{2}{75}=\frac{99}{200}<\frac{1}{2}.
\]

\textbf{Estimation 3.} Let $\operatorname{Re}\lambda\leq6.$\bigskip\ By (52)
and (49) we have%
\begin{equation}
\left\vert d_{4}\right\vert \leq\frac{\left\vert 2\times\frac{8}{\sqrt{6}%
}\right\vert ^{2}\left\vert d_{2}\right\vert }{\left\vert 43\right\vert
\left\vert 19\right\vert }\leq\frac{\left\vert 2\times\frac{8}{\sqrt{6}%
}\right\vert ^{2}\frac{\sqrt{2}}{2}}{\left\vert 43\right\vert \left\vert
19\right\vert },\text{ }\left\vert d_{5}\right\vert \leq\frac{\left\vert
2\times\frac{8}{\sqrt{6}}\right\vert ^{3}\left\vert d_{2}\right\vert
}{\left\vert 75\right\vert \left\vert 43\right\vert \left\vert 19\right\vert
}\leq\frac{\left\vert 2\times\frac{8}{\sqrt{6}}\right\vert ^{3}\frac{\sqrt{2}%
}{2}}{\left\vert 75\right\vert \left\vert 43\right\vert \left\vert
19\right\vert }.
\end{equation}
Now using (51) and (53) and taking into account $d(\lambda,\{6,7,...\})\leq
115$ we obtain%
\[
\left\vert d_{3}\right\vert ^{2}\leq\frac{\left\vert \frac{8}{\sqrt{6}%
}\right\vert ^{2}}{\left\vert 19\right\vert ^{2}}=\frac{32}{1083}\text{ }%
\And\text{ }\sum_{k=6}^{\infty}\left\vert d_{k}\right\vert ^{2}\leq
\frac{4\left\vert \frac{8}{\sqrt{6}}\right\vert ^{2}}{\left\vert
115\right\vert ^{2}}.
\]
These inequalities imply that
\[
\sum_{k=3}^{\infty}\left\vert d_{k}\right\vert ^{2}=\frac{32}{1083}+\left(
\frac{\left\vert 2\times\frac{8}{\sqrt{6}}\right\vert ^{2}\frac{\sqrt{2}}{2}%
}{\left\vert 43\right\vert \left\vert 19\right\vert }\right)  ^{2}+\left(
\frac{\left\vert 2\times\frac{8}{\sqrt{6}}\right\vert ^{3}\frac{\sqrt{2}}{2}%
}{\left\vert 75\right\vert \left\vert 43\right\vert \left\vert 19\right\vert
}\right)  ^{2}+\frac{4\left\vert \frac{8}{\sqrt{6}}\right\vert ^{2}%
}{\left\vert 115\right\vert ^{2}}<0.03\,\allowbreak415
\]

\textbf{Estimation 4. }\ Now we estimate $\frac{\left\vert d_{3}\right\vert
}{\left\vert d_{2}\right\vert }$ for $\operatorname{Re}\lambda\leq6$.
Iterating (45) for $k=3,$ we get%
\begin{equation}
d_{3}=\frac{ad_{2}+ad_{4}}{\lambda-25}=\frac{ad_{2}}{\lambda-25}+\frac
{a}{\lambda-25}(\frac{ad_{3}+ad_{5}}{\lambda-49})
\end{equation}%
\[
=\frac{ad_{2}}{\lambda-25}+\frac{a^{3}d_{2}}{(\lambda-25)^{2}(\lambda
-49)}+\frac{a^{3}d_{4}}{(\lambda-25)^{2}(\lambda-49)}+\frac{a^{2}d_{5}%
}{(\lambda-25)(\lambda-49)}.
\]
Therefore, dividing both sides of (72) by $d_{2}$ and using (52) we obtain
\[
\frac{\left\vert d_{3}\right\vert }{\left\vert d_{2}\right\vert }\leq
\frac{\frac{8}{\sqrt{6}}}{19}+\frac{\left\vert \frac{8}{\sqrt{6}}\right\vert
^{3}}{\left\vert 43\right\vert \left\vert 19\right\vert ^{2}}+\frac
{4\left\vert \frac{8}{\sqrt{6}}\right\vert ^{5}}{\left\vert 43\right\vert
^{2}\left\vert 19\right\vert ^{3}}+\frac{8\left\vert \frac{8}{\sqrt{6}%
}\right\vert ^{5}}{\left\vert 75\right\vert \left\vert 43\right\vert
^{2}\left\vert 19\right\vert ^{2}}\leq0.174\,32
\]

\textbf{Estimation 5. }Let $26\leq\operatorname{Re}\lambda\leq46.$ Using (56)
and (58) we obtain
\[
\left\vert b_{1}\right\vert ^{2}\leq\frac{\left\vert a\right\vert ^{2}%
}{\left\vert \lambda-4\right\vert ^{2}}\leq\frac{\left\vert 5\right\vert ^{2}%
}{\left\vert 22\right\vert ^{2}}=\frac{25}{484},\text{ }\left\vert
b_{2}\right\vert ^{2}\leq\frac{\left\vert a\right\vert ^{2}}{\left\vert
\lambda-16\right\vert ^{2}}\leq\frac{\left\vert 5\right\vert ^{2}}{\left\vert
10\right\vert ^{2}}=\frac{1}{4},
\]%
\[
\left\vert b_{4}\right\vert ^{2}\leq\frac{\left\vert a\right\vert ^{2}%
}{\left\vert \lambda-64\right\vert ^{2}}\leq\frac{\left\vert 5\right\vert
^{2}}{\left\vert 18\right\vert ^{2}}=\frac{25}{324},\text{ }\sum_{k=5}%
^{\infty}\left\vert b_{k}\right\vert ^{2}\leq\frac{4\left\vert 5\right\vert
^{2}}{\left\vert 54\right\vert ^{2}}=\frac{25}{729}.
\]
Thus
\[
\sum_{k\neq3}\left\vert b_{k}\right\vert ^{2}\leq\frac{25}{484}+\frac{1}%
{4}+\frac{25}{324}+\frac{25}{729}=\frac{145\,759}{352\,836}<\frac{1}{2}.
\]

\textbf{Estimation 6.} Let $16<\operatorname{Re}\lambda\leq26.$ By (55) and
(57) we have
\[
\left\vert b_{1}\right\vert ^{2}\leq\frac{\left\vert a\right\vert ^{2}%
}{\left\vert \lambda-4\right\vert ^{2}}\leq\frac{\left\vert 5\right\vert ^{2}%
}{\left\vert 12\right\vert ^{2}}=\frac{25}{144},\text{ }\left\vert
b_{3}\right\vert ^{2}\leq\frac{\left\vert a\right\vert ^{2}}{\left\vert
\lambda-36\right\vert ^{2}}\leq\frac{\left\vert 5\right\vert ^{2}}{\left\vert
10\right\vert ^{2}}=\frac{1}{4},
\]%
\[
\sum_{k=4}^{\infty}\left\vert b_{k}\right\vert ^{2}\leq\frac{4\left\vert
5\right\vert ^{2}}{\left\vert 38\right\vert ^{2}}=\frac{25}{361},\text{ }%
\sum_{k\neq2}\left\vert b_{k}\right\vert ^{2}\leq\frac{25}{144}+\frac{1}%
{4}+\frac{25}{361}=\frac{25\,621}{51\,984}<\frac{1}{2}.
\]

\textbf{Estimation 7.} Let$12<\operatorname{Re}\lambda\leq16.$ By (55) and
(57)
\[
\left\vert b_{1}\right\vert ^{2}\leq\frac{\left\vert 5\right\vert ^{2}%
}{\left\vert 8\right\vert ^{2}}=\frac{25}{64},\text{ }\left\vert
b_{3}\right\vert ^{2}\leq\frac{\left\vert 5\right\vert ^{2}}{\left\vert
20\right\vert ^{2}}=\frac{1}{16},\text{ }\left\vert b_{4}\right\vert ^{2}%
\leq\frac{\left\vert 5\right\vert ^{2}}{\left\vert 48\right\vert ^{2}}%
=\frac{25}{2304},
\]%
\[
\sum_{k=5}^{\infty}\left\vert b_{k}\right\vert ^{2}\leq\frac{4\left\vert
5\right\vert ^{2}}{\left\vert 84\right\vert ^{2}}=\frac{25}{1764},\text{ }%
\sum_{k\neq2}\left\vert b_{k}\right\vert ^{2}\leq\frac{25}{64}+\frac{1}%
{16}+\frac{25}{2304}+\frac{25}{1764}=\frac{53\,981}{112\,896}<\frac{1}{2}.
\]

\textbf{Estimation 8. }Let $\operatorname{Re}\lambda\leq12.$ By (55) and (57)
we have \textbf{ }%
\[
\left\vert b_{4}\right\vert ^{2}\leq\frac{\left\vert 5\right\vert ^{2}%
}{\left\vert 52\right\vert ^{2}}=\frac{25}{2704},\text{ }\left\vert
b_{3}\right\vert ^{2}\leq\frac{\left\vert 5\right\vert ^{2}}{\left\vert
24\right\vert ^{2}}=\frac{25}{576},\text{ }%
\]%
\[
\sum_{k=5}^{\infty}\left\vert b_{k}\right\vert ^{2}\leq\frac{4\left\vert
5\right\vert ^{2}}{\left\vert 88\right\vert ^{2}}=\frac{25}{1936},\text{ }%
\sum_{k=3}^{\infty}\left\vert b_{k}\right\vert ^{2}\leq\frac{25}{2704}%
+\frac{25}{576}+\frac{25}{1935}=\frac{30\,495}{465\,088}<\allowbreak\frac
{1}{15}.
\]

\textbf{Estimation 9. }Here we estimate $\frac{\left\vert b_{3}\right\vert
}{\left\vert b_{2}\right\vert }$ for $\operatorname{Re}\lambda\leq12.$
Iterating (43) for $k=3,$ we get\textbf{ }%
\begin{equation}
b_{3}=\frac{ab_{2}+ab_{4}}{\lambda-36}=\frac{ab_{2}}{\lambda-36}+\frac
{a}{\lambda-36}(\frac{ab_{3}+ab_{5}}{\lambda-64})
\end{equation}%
\[
=\frac{ab_{2}}{\lambda-36}+\frac{a^{3}b_{2}}{(\lambda-36)^{2}(\lambda
-64)}+\frac{a^{3}b_{4}}{(\lambda-36)^{2}(\lambda-64)}+\frac{a^{2}b_{5}%
}{(\lambda-36)(\lambda-64)}.
\]
Now dividing both sides of (73) by $b_{2}$ and using (56) we obtain
\[
\frac{\left\vert b_{3}\right\vert }{\left\vert b_{2}\right\vert }\leq\frac
{5}{24}+\frac{\left\vert 5\right\vert ^{3}}{\left\vert 52\right\vert
\left\vert 24\right\vert ^{2}}+\frac{4\left\vert 5\right\vert ^{5}}{\left\vert
52\right\vert ^{2}\left\vert 24\right\vert ^{3}}+\frac{8\left\vert
5\right\vert ^{5}}{\left\vert 88\right\vert \left\vert 52\right\vert
^{2}\left\vert 24\right\vert ^{2}}<0.213\,1
\]

\textbf{Estimation 10. }Let $3\leq\operatorname{Re}\lambda<8.$ By (42), Lemma
2($d$) and (55)
\[
\left\vert a_{0}\right\vert ^{2}\leq\frac{\left\vert \sqrt{2}aa_{1}\right\vert
^{2}}{\left\vert \lambda\right\vert ^{2}}\leq\frac{\left\vert \frac{4}%
{3}\right\vert ^{2}}{\left\vert 3\right\vert ^{2}}=\frac{16}{81},\text{
}\left\vert a_{2}\right\vert ^{2}\leq\frac{\left\vert a\right\vert ^{2}%
}{\left\vert \lambda-16\right\vert ^{2}}\leq\frac{\left\vert \frac{4}%
{3}\right\vert ^{2}}{\left\vert 8\right\vert ^{2}}=\frac{1}{36},
\]%
\[
\sum_{k=3}^{\infty}\left\vert a_{k}\right\vert ^{2}\leq\frac{4\left\vert
\frac{4}{3}\right\vert ^{2}}{\left\vert 28\right\vert ^{2}}=\frac{4}%
{441},\text{ }\sum_{k\neq1}\left\vert a_{k}\right\vert ^{2}\leq\frac{16}%
{81}+\frac{1}{36}+\frac{4}{441}<\frac{1}{2}.
\]

\textbf{Estimation 11. }Let $\operatorname{Re}\lambda<3.$ By Lemma 2($d$),
(55) and (57) we have\textbf{ }%
\[
\left\vert a_{2}\right\vert ^{2}\leq\frac{\left\vert a\right\vert ^{2}%
}{\left\vert \lambda-16\right\vert ^{2}}\leq\frac{\left\vert \frac{4}%
{3}\right\vert ^{2}}{\left\vert 13\right\vert ^{2}}=\frac{16}{1521},\text{
}\sum_{k=3}^{\infty}\left\vert a_{k}\right\vert ^{2}\leq\frac{4\left\vert
\frac{4}{3}\right\vert ^{2}}{\left\vert 33\right\vert ^{2}}=\frac{64}{9801},
\]%
\[
\sum_{k=2}^{\infty}\left\vert a_{k}\right\vert ^{2}\leq\frac{16}{1521}%
+\frac{64}{9801}<\frac{1}{58}.
\]

\textbf{Estimation 12. }Here we estimate $\frac{a_{2}}{a_{1}}$ for
$\operatorname{Re}\lambda<3$. Iterating (42) for $k=2,$ we get%
\begin{equation}
a_{2}=\frac{aa_{1}+aa_{3}}{\lambda-16}=\frac{aa_{1}}{\lambda-16}+\frac
{a}{\lambda-16}(\frac{aa_{2}+aa_{4}}{\lambda-36})
\end{equation}%
\[
=\frac{aa_{1}}{\lambda-16}+\frac{a^{3}a_{1}}{(\lambda-16)^{2}(\lambda
-36)}+\frac{a^{3}a_{3}}{(\lambda-16)^{2}(\lambda-36)}+\frac{a^{2}a_{4}%
}{(\lambda-16)(\lambda-36)}.
\]
Now dividing both sides of (74) by $a_{1}$ and using Lemma 2($d$), (56) we
obtain
\[
\frac{\left\vert a_{2}\right\vert }{\left\vert a_{1}\right\vert }\leq
\frac{\frac{4}{3}}{13}+\frac{\left\vert \frac{4}{3}\right\vert ^{3}%
}{\left\vert 33\right\vert \left\vert 13\right\vert ^{2}}+\frac{4\left\vert
\frac{4}{3}\right\vert ^{5}}{\left\vert 33\right\vert ^{2}\left\vert
13\right\vert ^{3}}+\frac{8\left\vert \frac{4}{3}\right\vert ^{5}}{\left\vert
61\right\vert \left\vert 33\right\vert ^{2}\left\vert 13\right\vert ^{2}%
}<0.103\,01
\]

\end{document}